\documentclass[11pt]{amsart}

\usepackage[T1]{fontenc}
\usepackage{lmodern}
\usepackage{amsmath,amssymb,mathtools, mathrsfs}
\usepackage[margin=1.12in]{geometry}
\usepackage{microtype}
\usepackage{xcolor}
\usepackage[colorlinks=true,linkcolor=blue!45!black,citecolor=blue!45!black,urlcolor=blue!45!black]{hyperref}

\usepackage{tikz}
\usetikzlibrary{arrows.meta}

\newtheorem{theorem}{Theorem}[section]
\newtheorem{proposition}[theorem]{Proposition}
\newtheorem{lemma}[theorem]{Lemma}
\newtheorem{corollary}[theorem]{Corollary}
\newtheorem{conjecture}[theorem]{Conjecture}

\theoremstyle{definition}

\theoremstyle{remark}

\newcommand{\N}{\mathbb N}
\newcommand{\Z}{\mathbb Z}
\newcommand{\C}{\mathbb C}

\newcommand{\ind}{\mathbf 1}
\newcommand{\Dpret}{\mathbb D}
\newcommand{\Mpret}{\mathcal M}
\newcommand{\ee}{\mathrm e}
\newcommand{\dd}{\,\mathrm d}
\newcommand{\oh}{{\mathrm o}}

\newcommand{\norm}[1]{\left\lVert#1\right\rVert}
\DeclareMathOperator{\Rea}{Re}

\numberwithin{equation}{section}

\author[B. Wang]{Biao Wang}
\address{School of Mathematics and Statistics, Yunnan University, Kunming, Yunnan 650500, China}
\email{bwang@ynu.edu.cn}

\author[S. Yi]{Shaoyun Yi}
\address{School of Mathematical Sciences, Xiamen University, Xiamen, Fujian 361005, China}
\email{yishaoyun926@xmu.edu.cn}

\title[Dynamical generalizations of the two-point Chowla conjecture]
{Dynamical generalizations of the two-point Chowla conjecture}
\date{\today}
\subjclass[2020]{11N37, 37A44}

\keywords{prime Omega function, Chowla conjecture, unique ergodicity}

\begin{document}

\begin{abstract}
In 1965, Chowla conjectured that the values of the Liouville function become asymptotically uncorrelated at consecutive natural numbers. Recently, the two-point case is resolved by OpenAI. In 2026, Charamaras and Richter conjectured a generalization of the Chowla conjecture to the asymptotic independence of the consecutive prime Omega functions. In this article, we prove the two-point case of their conjecture. As a consequence, we obtain a dynamical generalization of the two-point Chowla conjecture.
\end{abstract}

\maketitle

\section{Introduction and the main result}

Let $\Omega(n)$ be the number of prime factors of $n$
counted with multiplicity, and put $\Omega(1)=0$. 
We also write $\omega(n)=\sum_{p\mid n}1$ for the number of distinct
prime factors, and $\omega_{\mathcal R}(n)=\sum_{p\in\mathcal R,\,p\mid n}1$
for the count restricted to a finite prime set $\mathcal R$.
Throughout, $\N=\{0, 1,2,\ldots\}$.
Write $\lambda(n)=(-1)^{\Omega(n)}$ for the Liouville function.
It is well-known (e.g., \cite{Landau1909}) that the prime number theorem (PNT) is equivalent to the assertion that
\begin{equation}\label{pnt_Liouville}
	\lim_{N\to\infty}\frac1N\sum_{n=1}^N\lambda(n)=0.
\end{equation}

Let $k\ge1$ be a natural number. In 1965, Chowla \cite{Chowla1965} gave the following conjecture on the $k$-point correlations of the Liouville function: 
		for any sequence of $k$ distinct non-negative integers  $h_1, \ldots, h_k$, we have
		\begin{equation}\label{Chowla_conjecture_eqn}
			\lim_{N\to\infty}\frac1N\sum_{n=1}^N \lambda(n+h_1)\ldots\lambda(n+h_k) = 0.
		\end{equation}
For the case $k=1$, Chowla's conjecture \eqref{Chowla_conjecture_eqn} is the same as  \eqref{pnt_Liouville}. For $k=2$, Tao \cite{Tao2016} proved  in 2016 the logarithmic average of  \eqref{Chowla_conjecture_eqn} over $n$  that 
\begin{equation}\label{Tao2016}
	\frac{1}{\log N}\sum_{n=1}^N \frac{\lambda(n)\lambda(n+1)}{n}=o_{N\to\infty}(1).
\end{equation}
The error term was improved to $O(\frac1{\sqrt{\log\log N}})$ by Helfgott and Radziwi\l\l{} \cite{HelfgottRadziwill2021} and later to $O(\frac1{(\log N)^c})$ by Pilatte \cite{Pilatte2026} for some constant $c>0$. Recently, the two-point Chowla's conjecture  is resolved by OpenAI \cite{OAI}, that is, we have
\begin{equation}\label{AI2026}
	\lim_{N\to\infty} \frac{1}{ N}\sum_{n=1}^N \lambda(n)\lambda(n+1)=0.
\end{equation}

 In 2024, Charamaras and Richter \cite{CharamarasRichter2025} generalized \eqref{Tao2016} to bounded functions over $\Omega(n)$. And they made a conjecture that generalizes  \eqref{Chowla_conjecture_eqn} as follows.
\begin{conjecture}[{\cite[Conjecture 1.10]{CharamarasRichter2025}}]
\label{conj_functional_chowla_str}
For any positive $k\in\N$ and any bounded $a\colon\N^k\to\C$, we have
$$
\frac1N\sum_{n=1}^N a\big(\Omega(n),\Omega(n+1),\dots,\Omega(n+k-1)\big)
= \frac{1}{N^{k}} \sum_{1\leq n_1,\dots, n_k\leq N}  a\big(\Omega(n_1),\dots,\Omega(n_k)\big) + \oh_{N\to\infty}(1).
$$
\end{conjecture}

In this article, we prove the case $k=2$ of Conjecture~\ref{conj_functional_chowla_str}.

\begin{theorem}\label{thm:main}
Let $a_1,a_2,b_1,b_2$ be fixed positive integers with
$a_1b_2-a_2b_1\ne0$. Uniformly for any bounded function
$a:\N^2\to\C$, as $N\to\infty$, we have
\begin{align}\label{eq:main-test}
 \frac1N\sum_{n\le N}
 a\bigl(\Omega(a_1n+b_1),\Omega(a_2n+b_2)\bigr)
 ={}&\frac1{N^2}\sum_{m,n\le N}a(\Omega(m),\Omega(n))\notag\\
 &+O_{a_1,a_2,b_1,b_2}\!\left(
 \norm a_\infty\frac{(\log\log\log N)^{5/2}}
                         {\sqrt{\log\log N}}\right).
\end{align}

\end{theorem}

Let $(X,\nu,T)$ be a uniquely ergodic topological dynamical system.
Let $C(X)$ denote the continuous complex-valued functions on $X$. In 2022, Bergelson and Richter \cite[Theorem A]{BR} proved a dynamical generalization of \eqref{pnt_Liouville} that
\begin{equation}\label{eq:BR}
\lim_{N\to\infty} \frac1N\sum_{n\leq N}f(T^{\Omega(n)}x)=\int_X f\,d\nu
\end{equation}
for every $f\in C(X)$ and $x\in X$.  By \cite[Theorem 1.7]{Wang2026avg_1}, as a consequence of Theorem~\ref{thm:main}, we have the following dynamical analogue of \eqref{eq:BR} for the two-point Chowla conjecture \eqref{AI2026}.

\begin{theorem}\label{dynamical_chowla}
In a uniquely ergodic system $(X,\nu,T)$, we have
\begin{equation}
\lim_{N\to\infty} \frac1N\sum_{n\leq N}f(T^{\Omega(n)+\Omega(n+1)}x)=\int_X f\,d\nu
\end{equation}
for every $f\in C(X)$ and $x\in X$. 
\end{theorem}

Theorem~\ref{dynamical_chowla} also holds for $\Omega(a_1n+b_1)+\Omega(a_2n+b_2)$ with $a_1b_2-a_2b_1\ne0$. To prove Theorem~\ref{thm:main}, we will use a Fourier analytic method of Qi and Zheng \cite{QiZheng2026}, which is used to establish a variant of \eqref{eq:BR} for irreducible cubic forms.  We prove a stronger $\ell^1$ estimate \eqref{eq:main-tv} for the
local densities defined below. Then we divide  the square integral into two parts according to the frequencies. For the large frequencies, we will use a 
uniform pretentious-distance estimate in \cite{Wang} and the results in \cite{OAI}. For the small frequencies, we will estimate the integral directly.

\section{Fourier reduction}
\label{sec:fourier}

For positive integers $a_1,a_2,b_1,b_2$  with
$a_1b_2-a_2b_1\ne0$, we write
\[
 \ell_i(n)=a_in+b_i \, (i=1,2),\qquad
 \delta=a_1b_2-a_2b_1\ne0,\qquad L=\log\log N.
\]
For $k,j\in\N$, we define the local densities
\begin{align}
 p_N^\ell(k,j)
 &=\frac1N\#\{n\le N:\Omega(\ell_1(n))=k,
                         \ \Omega(\ell_2(n))=j\},\label{eq:joint}\\
 q_N(k)&=\frac1N\#\{n\le N:\Omega(n)=k\},\label{eq:marginal}\\
 d_N(k,j)&=p_N^\ell(k,j)-q_N(k)q_N(j).\notag
\end{align}
Extend these functions by zero to negative indices.
By \cite[Conjecture 1.9]{CharamarasRichter2025}, to prove Theorem~\ref{thm:main}, it suffices to prove  the following estimate on the total variation of $d_N(k,j)$.  

\begin{theorem}\label{thm_main_goal}
We have that
\begin{equation}\label{eq:main-tv}
 \sum_{k,j\in\Z}|d_N(k,j)|
 \ll_{\ell_1,\ell_2}\frac{(\log L)^{5/2}}{\sqrt L}.
\end{equation}
\end{theorem}

For $\theta,\phi\in[-\pi,\pi]$, let
\begin{align}
 f_\theta(n)&=\exp(i\theta\Omega(n)),\notag\\
 A_N(\theta)&=\frac1N\sum_{n\le N}f_\theta(n),\notag\\
 C_N^\ell(\theta,\phi)
 &=\frac1N\sum_{n\le N}f_\theta(\ell_1(n))f_\phi(\ell_2(n)),\notag\\
 \Delta_N(\theta,\phi)
 &=C_N^\ell(\theta,\phi)-A_N(\theta)A_N(\phi).
 \label{eq:delta}
\end{align}

The Fourier transform $\widehat d_N$ of $d_N$ is given by
\[
\widehat d_N(\theta,\phi)
=\sum_{k,j\in\mathbb Z}d_N(k,j)e^{i(k\theta+j\phi)}.
\]
By the definition of $p_N^\ell(k,j)$,
\[
\sum_{k,j\in\mathbb Z}p_N^\ell(k,j)e^{i(k\theta+j\phi)}=\frac1N\sum_{n\le N} e^{i\theta\Omega(\ell_1(n))+i\phi\Omega(\ell_2(n))}=C_N^\ell(\theta,\phi).
\]
Similarly, 
\[
\sum_{k\in\mathbb Z}q_N(k)e^{ik\theta}
=\frac1N\sum_{n\le N}e^{i\theta\Omega(n)}
=A_N(\theta).
\]
Then
\[
\sum_{k,j\in\mathbb Z}q_N(k)q_N(j)e^{i(k\theta+j\phi)}=\left(\sum_{k\in\mathbb Z}q_N(k)e^{ik\theta}\right)
  \left(\sum_{j\in\mathbb Z}q_N(j)e^{ij\phi}\right)=A_N(\theta)A_N(\phi).
\]
It follows that
\[
\widehat d_N(\theta,\phi)
=C_N^\ell(\theta,\phi)-A_N(\theta)A_N(\phi)
=\Delta_N(\theta,\phi).
\]
That is, $\Delta_N$ is the Fourier transform of $d_N$. By Parseval's formula,
\begin{equation}\label{eq:parseval}
 \sum_{k,j\in\Z}|d_N(k,j)|^2
 =\frac1{(2\pi)^2}\int_{[-\pi,\pi]^2}
 |\Delta_N(\theta,\phi)|^2\dd\theta\dd\phi.
\end{equation}

For $R\ge1$ and $L=\log\log N$, put
\begin{align*}
 I_R&=\{k\in\Z:|k-L|\le R\},\\
 \tau_N(R)&=\frac1N\#\{n\le N:|\Omega(n)-L|>R\},\\
 \tau_{N,i}^\ell(R)
 &=\frac1N\#\{n\le N:|\Omega(\ell_i(n))-L|>R\}\quad(i=1,2).
\end{align*}

\begin{lemma}\label{lem:fourier-reduction}
Define  the $L^2$ norm
$$\|\Delta_N\|_{L^2([-\pi,\pi]^2)}:= \Big(\int_{[-\pi,\pi]^2}
 |\Delta_N(\theta,\phi)|^2\dd\theta\dd\phi\Big)^{\frac12}.$$ Then for every $R\ge1$, we have
\begin{equation}\label{eq:fourier-reduction}
 \sum_{k,j}|d_N(k,j)|
 \ll R\norm{\Delta_N}_{L^2([-\pi,\pi]^2)}
       +\tau_N(R)+\tau_{N,1}^\ell(R)+\tau_{N,2}^\ell(R).
\end{equation}
\end{lemma}

\begin{proof}
We split the sum into two parts:
\[
\sum_{k,j\in\mathbb Z}|d_N(k,j)|
=
\sum_{(k,j)\in I_R^2}|d_N(k,j)|
+
\sum_{(k,j)\notin I_R^2}|d_N(k,j)|.
\]

If $(k,j)\in I_R^2$, then by $|I_R|\le 2R+1\le3R$, the Cauchy--Schwarz inequality gives
\[
\sum_{(k,j)\in I_R^2}|d_N(k,j)| \le |I_R^2|^{1/2}
   \left(\sum_{(k,j)\in I_R^2}|d_N(k,j)|^2\right)^{1/2}\le 3R
   \left(\sum_{k,j\in\mathbb Z}|d_N(k,j)|^2\right)^{1/2}.
\]
It follows by \eqref{eq:parseval} that
\begin{equation}\label{parseval_central}
\sum_{(k,j)\in I_R^2}|d_N(k,j)|
\le
\frac{3R}{2\pi}
\|\Delta_N\|_{L^2([-\pi,\pi]^2)}.
\end{equation}

If $(k,j)\notin I_R^2$, then $k\notin I_R$ or $ j\notin I_R$. So
\[
\begin{aligned}
\sum_{(k,j)\notin I_R^2}p_N^\ell(k,j)
&\le
\sum_{k\notin I_R}\sum_j p_N^\ell(k,j)
+
\sum_{j\notin I_R}\sum_k p_N^\ell(k,j)\\
&=
\frac1N\#\{n\le N:|\Omega(\ell_1(n))-L|>R\}\\
&\quad+
\frac1N\#\{n\le N:|\Omega(\ell_2(n))-L|>R\}\\
&=\tau_{N,1}^\ell(R)+\tau_{N,2}^\ell(R).
\end{aligned}
\]
Similarly, using $\sum_k q_N(k)=1$, we have
\[
\sum_{(k,j)\notin I_R^2}q_N(k)q_N(j) 
\le
\left(\sum_{k\notin I_R}q_N(k)\right)
\left(\sum_jq_N(j)\right)+
\left(\sum_kq_N(k)\right)
\left(\sum_{j\notin I_R}q_N(j)\right)=2\tau_N(R).
\]

By the triangle inequality,
\begin{equation}\label{parseval_tail}
\sum_{(k,j)\notin I_R^2}|d_N(k,j)| \leq 2\tau_N(R)
+\tau_{N,1}^\ell(R)+\tau_{N,2}^\ell(R).
\end{equation}

Thus, combining \eqref{parseval_central} and \eqref{parseval_tail} yields
\[
\sum_{k,j\in\mathbb Z}|d_N(k,j)|
\le
\frac{3R}{2\pi}\|\Delta_N\|_{L^2([-\pi,\pi]^2)}
+2\tau_N(R)
+\tau_{N,1}^\ell(R)+\tau_{N,2}^\ell(R).
\]
This implies \eqref{eq:fourier-reduction}, with an absolute
implied constant.
\end{proof}

\section{Uniform correlations}

For $1$-bounded multiplicative functions $f,g$, define
\[
 \Dpret(f,g;Y)^2=\sum_{p\le Y}\frac{1-\Rea(f(p)\overline{g(p)})}{p},
\]
and put
\[
 \Mpret(f;Y,Q)=
 \inf_{\substack{q\le Q,\ \chi\bmod q\\ |t|\le Y}}
 \Dpret(f,\chi n^{it};Y)^2.
\]
Here the infimum ranges over all Dirichlet characters of the indicated
moduli, including principal and imprimitive characters. Then by \cite[Lemma~2.2]{Wang}, we have a uniform distance estimate for  $\Mpret(f_\theta;Y,Q)$.

\begin{lemma}[{\cite[Lemma~2.2]{Wang}}]\label{lem:distance}
There are absolute constants $c,C>0$ such that, for all sufficiently
large $Y$, all $\theta\in[-\pi,\pi]$, and
$1\le Q\le(\log Y)^{1/125}$,
\begin{equation}\label{eq:distance}
 \Mpret(f_\theta;Y,Q)\ge c\theta^2\log\log Y-C.
\end{equation}
\end{lemma}

We will also use the short exponential-sum estimate proved in \cite[Theorem~1.7]{MRT}. 

\begin{theorem}[{\cite[Theorem~1.7]{MRT}}] \label{thm_MRT}
If $Y\ge H\ge10$ and $f$ is a $1$-bounded
multiplicative function, then
\begin{align}\label{eq:mrt}
 \sup_{\theta\in\mathbb R}\int_0^Y
 \left|\sum_{x<m\le x+H}f(m)\ee^{2\pi i\theta m}\right|\dd x
 \ll YH\left(
 \ee^{-\Mpret(f;Y,Q_H)/20}
 +\frac{\log\log H}{\log H}
 +(\log Y)^{-1/700}\right),
\end{align}
where $Q_H=\min\{(\log Y)^{1/125},(\log H)^5\}$.
\end{theorem}

In \cite[Proposition 2.1]{OAI}, it is proved that for any fixed $h,q\geq1$ and $r\in\Z$, we have
\begin{equation}\label{Chowla_AP}
\frac1N\sum_{\substack{n\le N\\n\equiv r\pmod q}} \lambda(n)\lambda(n+h) \ll_{h,q} (\log N)^{-c}.
\end{equation}
In this section, we will use the method in the proof of \eqref{Chowla_AP} to estimate correlations of $f_\theta(n)$ and $f_\phi(n)$ in arithmetic progressions.

\begin{proposition}\label{prop:progression} 
Let $N\ge3$. Fix integers $h,q\ge1$. For $r\in\Z$ and
$\theta,\phi\in[-\pi,\pi]$, put
\[
 F_N^{h,q,r}(\theta,\phi)
 =\frac1N\sum_{\substack{n\le N\\n\equiv r\pmod q}}
 f_\theta(n)f_\phi(n+h).
\]
There are absolute positive constants $A_0,b,c$ and a constant
$G_0=G_0(h,q)$ such that
\begin{equation}\label{eq:progression-phase-bound}
 |F_N^{h,q,r}(\theta,\phi)|
 \ll_{h,q}G^{-b}
       +G\exp\{-c\max(\theta^2,\phi^2)\log\log N\}
\end{equation}
holds uniformly for $r\in \Z$, $\theta,\phi\in[-\pi,\pi]$ and
\begin{equation}\label{eq:scale-range}
 G\ge G_0,\qquad \log N\ge G^{A_0}.
\end{equation}
\end{proposition}
 If taking $\theta=\phi=\pi$ in \eqref{eq:progression-phase-bound}, then we recover \eqref{Chowla_AP}. 
To prove Proposition~\ref{prop:progression}, we need to establish some lemmas. We set up some notations first. Let $N\ge3$ and  $L=\log\log N$. Throughout this setup and the next section, fix positive
integers $h,q$, an integer $r$, phases $\theta,\phi\in[-\pi,\pi]$,
and a real number $G>1$ with $\log N\ge G^{A_0}$.
The absolute constant $A_0$ will be chosen sufficiently large below.

Fix $\delta_0=1/200$ and let $W$ be a sufficiently large fixed
constant. Set
\begin{equation}\label{eq:graph-parameters}
 b=\frac{\delta_0}{6W},\qquad
 J=\lfloor b\log G\rfloor,\qquad
 \eta=\ee^{-J},\qquad K=\ee^{4J}.
\end{equation}
Let
\begin{gather*}
 Y_i=G^{1-\delta_0}\ee^{6Wi}\quad(0\le i\le J),\\
 \mathcal P_i=\{p:p\equiv1\pmod5,\ p\nmid hq,\ Y_{i-1}\le\log p<Y_i\},
 \qquad 1\le i\le J,
\end{gather*}
and put $\mathcal P=\bigcup_i\mathcal P_i$ and
$\mathcal Q=\{p\le\ee^G:p\not\equiv1\pmod5,\ p\nmid hq\}$.
Let
\[
 \mathcal D=\{p_1\cdots p_J:p_i\in\mathcal P_i\},\qquad
 V_i=\sum_{p\in\mathcal P_i}\frac1p,\qquad V_* =\prod_{i=1}^J V_i.
\]

For large $G$, we have $\mathcal P_i=\{p:p\equiv1\pmod5,\ Y_{i-1}\le\log p<Y_i\}$. By Mertens' estimate in arithmetic progressions, there is a constant $c_{0}$ such that
\[
\sum_{\substack{p\le x\\p\equiv1\pmod5}}\frac1p
=\frac14\log\log x+c_{0}+o(1),
\qquad x\to\infty.
\]
Then we have
\[
V_i=\sum_{p\in\mathcal P_i}\frac1p=\frac14\bigl(\log Y_i-\log Y_{i-1}\bigr)+o(1)=\frac14\log(e^{6W})+o(1)=\frac32W+o(1).
\]
So we have
$W\le V_i\le2W$ for all $1\leq i\leq J$.

By \(
J=\left\lfloor\frac{\delta_0}{6W}\log G\right\rfloor
\), we have \(
6WJ\le\delta_0\log G\) and \(Y_J=G^{1-\delta_0}e^{6WJ}\le G.\) Then
\[
\log d =\sum_{i=1}^J\log p_i
<\sum_{i=1}^JY_i =Y_J\sum_{a=0}^{J-1}e^{-6Wa} <\frac{Y_J}{1-e^{-6W}}
\le\frac{G}{1-e^{-6W}}
<2G
\]
for sufficiently large $W$. So $d<\ee^{2G}$ for $d\in\mathcal D$.

Let $\mathcal V$ be the squarefree products of primes in
$\mathcal Q$, including $1$. Since $p\nmid hq$ for any $p\in \mathcal P_i$, we have $(dv,hq)=1$ for any  $d \in \mathcal D$ and $v\in \mathcal V$.
Recall that
$\omega_{\mathcal Q}(n)=\sum_{p\in\mathcal Q}\ind_{p\mid n}$
for $\mathcal Q$. Here $\ind_P$ is equal to 1 if the statement $P$ is true and zero otherwise.
Set
\begin{gather*}
 u(v)=4^{\omega(v)},\qquad
 w(n)=\sum_{\substack{v\in\mathcal V\\v\mid n}}u(v)
     =5^{\omega_{\mathcal Q}(n)},\qquad
 S=\prod_{p\in\mathcal Q}\left(1+\frac4p\right).
\end{gather*}

For integers $j\ge0$ with $j\eta\le100G$, let
\[
 \mathcal S_j=\{(d,v)\in\mathcal D\times\mathcal V:
       \omega(v)\le100\log G,\ j\eta\le\log(dv)<(j+1)\eta\},
\]
and put $T_j=N\ee^{j\eta}$. There are $O(G/\eta)$ many such $\mathcal S_j$.
Define
\[
 S_0=\sum_j\sum_{(d,v)\in\mathcal S_j}\frac{u(v)}{vd},
\]
then $SV_*/2\le S_0\le SV_*$ by \cite[Lemma~2.2]{OAI}.

Choose $\theta_*\in\{\theta,\phi\}$ with
$|\theta_*|=\max(|\theta|,|\phi|)$, and define
\begin{equation}\label{eq:pretentious-scale}
 M_* =\inf_{N/3\le Y\le N\exp(110G)}
 \Mpret\left(f_{\theta_*};Y,
       \min\{(\log Y)^{1/125},(10G)^5\}\right).
\end{equation}
For the short exponential sum errors in Theorem~\ref{thm_MRT}, let
\begin{equation}\label{eq:epsilon-star}
 \varepsilon_* =\ee^{-M_*/20}
       +G^{-0.995}\log G+(\log N)^{-1/700}.
\end{equation}

\begin{lemma}\label{lem:rough-shifts}
Use the common setup above. Let $H$ be an integer with
\[
 \tfrac12\exp(G^{0.995})\le H\le\exp(2G),
 \qquad N\ee^{-\eta}H<Y\le N\exp(2G).
\]
Let $\mathcal Z\subset[H,2H)$ consist of positive integers all of
whose prime factors belong to $\mathcal P$. For $|c_z|\le1$ and
any residue $r_0\pmod q$, put
\[
 \mathcal R=\frac1Y\sum_{n\le Y}\sum_{z\in\mathcal Z}
 \frac{c_z}{z}\ind_{n\equiv r_0\pmod q}
 f_\theta(n)f_\phi(n+hz).
\]
Then, uniformly in $\mathcal Z,c_z,r_0,\theta,\phi$,
\begin{equation}\label{eq:rough-shift-bound}
 |\mathcal R|\ll_h G^{-21/20}
           +G^{-3/4}\varepsilon_*+H/Y.
\end{equation}
\end{lemma}

\begin{proof}
The lower bound for $Y$ implies $(2h+1)H\le Y$ for sufficiently
large $G$, and $H/Y\le\ee^\eta/N$. Thus both short-window lengths
used below satisfy the size requirement in \eqref{eq:mrt}.
Put
\[
 B(t)=\sum_{z\in\mathcal Z}\frac{c_z}{z}\ee^{2\pi i hzt}, t\in[0,1].
\]
By \cite[Lemma~4.2]{OAI}, we have
\begin{equation}\label{eq:rough-fourier}
 \norm B_\infty\ll G^{-0.99},\qquad
 \int_0^1|B(t)|^4\dd t\ll H^{-1}G^{-3.96}.
\end{equation}

We give the details of the translation average and the Fourier
identity.  Set
\[
 F(n)=
 \sum_{z\in\mathcal Z}\frac{c_z}{z}
 \mathbf 1_{n\equiv r_0\pmod q}
 f_\theta(n)f_\phi(n+hz),
\]
so that
\[
 \mathcal R=\frac1Y\sum_{1\le n\le Y}F(n).
\]
Since $\mathcal Z\subset[H,2H)\cap\mathbb N$ and $|c_z|\le1$, we have
\[
 |F(n)|
 \le\sum_{z\in\mathcal Z}\frac1z
 \le\frac{\#\mathcal Z}{H}
 \le1.
\]

Put $M=\lfloor Y\rfloor$. For each fixed integer $1\le a\le H$,
we have
\[
 \sum_{n=1}^{M}F(n)
 =
 \sum_{x=1}^{M}F(x+a)+E_a,
 \qquad
 E_a=
 \sum_{n=1}^{a}F(n)-\sum_{n=M+1}^{M+a}F(n).
\]
Since $|F(n)|\le1$, we have $|E_a|\le2a$. Summing this identity over $a=1,\ldots,H$, 
\[
 H\sum_{n=1}^{M}F(n)
 =
 \sum_{a=1}^{H}\sum_{x=1}^{M}F(x+a)
 +\sum_{a=1}^{H}E_a.
\]
Dividing by $YH$ and recalling that
$\mathcal R=Y^{-1}\sum_{n=1}^{M}F(n)$, we deduce
\[
 \mathcal R
 =
 \frac1{YH}\sum_{1\le x\le Y}\sum_{a=1}^{H}F(x+a)
 +\frac1{YH}\sum_{a=1}^{H}E_a.
\]
The error satisfies
\[
 \left|\frac1{YH}\sum_{a=1}^{H}E_a\right|
 \le\frac2{YH}\sum_{a=1}^{H}a=\frac{H+1}{Y}
 =O(H/Y).
\]
Therefore,
\begin{align*}
 \mathcal R
 &=
 \frac1{YH}\sum_{1\le x\le Y}\sum_{a=1}^{H}F(x+a)
 +O(H/Y)\\
 &=
 \frac1{YH}\sum_{1\le x\le Y}\sum_{a=1}^{H}
 \sum_{z\in\mathcal Z}\frac{c_z}{z}
 \mathbf 1_{x+a\equiv r_0\pmod q}
 f_\theta(x+a)f_\phi(x+a+hz)
 +O(H/Y).
\end{align*}

Now write $e(t)=e^{2\pi it}$ and define
\[
\begin{aligned}
 U_x(t)&=\sum_{a=1}^{H}
 f_\theta(x+a)\mathbf 1_{x+a\equiv r_0\pmod q}e(at),\\
 V_x(t)&=\sum_{b=1}^{(2h+1)H}f_\phi(x+b)e(-bt).
\end{aligned}
\]
Then expanding the sums yields
\begin{align*}
 \int_0^1 B(t)U_x(t)V_x(t)\,dt
 &=
 \sum_{z\in\mathcal Z}\sum_{a=1}^{H}
 \sum_{b=1}^{(2h+1)H}
 \frac{c_z}{z}
 \mathbf 1_{x+a\equiv r_0\pmod q}
 f_\theta(x+a)f_\phi(x+b)\\
 &\hspace{3em}\times
 \int_0^1e\bigl((hz+a-b)t\bigr)\,dt\\
 &=
 \sum_{a=1}^{H}\sum_{z\in\mathcal Z}\frac{c_z}{z}
 \mathbf 1_{x+a\equiv r_0\pmod q}
 f_\theta(x+a)f_\phi(x+a+hz).
\end{align*}
Substituting this identity into above translation average gives
\begin{equation}\label{eq:rough-convolution}
 \mathcal R
 =
 \frac1{YH}\sum_{1\le x\le Y}
 \int_0^1B(t)U_x(t)V_x(t)\,dt
 +O(H/Y).
\end{equation}

Since $|f_\theta(n)|=|f_\phi(n)|=1$, Parseval's identity gives
\[
 \int_0^1|U_x(t)|^2\,dt
 =
 \sum_{a=1}^{H}
 \left|f_\theta(x+a)\mathbf 1_{x+a\equiv r_0\pmod q}\right|^2=
 \sum_{a=1}^{H}\mathbf 1_{x+a\equiv r_0\pmod q}
 \le H,
\]
and
\[
 \int_0^1|V_x(t)|^2\,dt
 =
 \sum_{b=1}^{(2h+1)H}|f_\phi(x+b)|^2
 =(2h+1)H.
\]
Applying the Cauchy--Schwarz inequality, we  obtain
\[
\begin{aligned}
 \int_0^1|U_x(t)V_x(t)|\,dt
 &\le
 \left(\int_0^1|U_x(t)|^2\,dt\right)^{1/2}
 \left(\int_0^1|V_x(t)|^2\,dt\right)^{1/2}\\
 &\le H^{1/2}\bigl((2h+1)H\bigr)^{1/2}=\sqrt{2h+1}\,H
 \ll_h H.
\end{aligned}
\]

Thus, on $|B(t)|\le G^{-21/20}$, this part of
\eqref{eq:rough-convolution} is $O_h(G^{-21/20})$. Let $\mathcal T=\{t\in[0,1]: |B(t)|> G^{-21/20} \}$ be the complementary set. By \eqref{eq:rough-fourier}, it has measure
\[
 |\mathcal T|\ll G^{84/20}H^{-1}G^{-3.96}
              =H^{-1}G^{0.24}.
\]

Let
\[
 \mathcal E_{\mathcal T}
 :=
 \frac1{YH}\sum_{1\le x\le Y}
 \int_{\mathcal T}B(t)U_x(t)V_x(t)\,dt.
\]
Suppose that 
\begin{equation}\label{upper_bound_W}
 \sup_{t\in[0,1]}\sum_{1\le x\le Y}|W_x(t)|
 \ll_h \varepsilon_*YH,
\end{equation}
where $W_x$ is the factor corresponding to the phase $\theta_*$:
we take $W_x=U_x$ if $\theta_*=\theta$, and $W_x=V_x$ otherwise.
By the trivial bound $|U_x(t)|\le H$ and $|V_x(t)|\le(2h+1)H$, we have
\[
 \sum_{1\le x\le Y}|U_x(t)V_x(t)|
 \le(2h+1)H\sum_{1\le x\le Y}|W_x(t)|
 \ll_h\varepsilon_* YH^2.
\]
Then by the triangle inequality and $ \norm B_\infty\ll G^{-0.99}$ in \eqref{eq:rough-fourier}, we obtain
\[
\begin{aligned}
 |\mathcal E_{\mathcal T}|
 &\le
 \frac1{YH}\int_{\mathcal T}|B(t)|
 \sum_{1\le x\le Y}|U_x(t)V_x(t)|\,dt\ll_h
 \frac1{YH}\,G^{-0.99}\,
 \varepsilon_* YH^2\,|\mathcal T|\\
 &\ll_h
 \varepsilon_* H G^{-0.99}
 \bigl(H^{-1}G^{0.24}\bigr)=\varepsilon_* G^{-3/4}.
\end{aligned}
\]

Combining these two bounds with the translation error
$O(H/Y)$ in \eqref{eq:rough-convolution}, we conclude that
\[
 |\mathcal R|
 \ll_h
 G^{-21/20}+G^{-3/4}\varepsilon_*+\frac HY,
\]
as asserted in \eqref{eq:rough-shift-bound}. Therefore, it remains to justify the estimate \eqref{upper_bound_W}.

Recall that $M=\lfloor Y\rfloor$. Put $X=M+1$. For an integer  
$L_0\in\{H,(2h+1)H\}$, define
\[
 S_{x,L_0}(\alpha)
 =
 \sum_{j=1}^{L_0}f_{\theta_*}(x+j)e(\alpha j),
 \qquad e(t)=e^{2\pi it}.
\]
For an integer $x$ and $u\in[x,x+1)$, we have
\[
 \sum_{u<n\le u+L_0}f_{\theta_*}(n)e(\alpha n)
 =
 e(\alpha x)S_{x,L_0}(\alpha).
\]
Consequently,
\begin{align*}
 \sum_{x=1}^{M}|S_{x,L_0}(\alpha)|
 =
 \int_1^{M+1}
 \left|
 \sum_{u<n\le u+L_0}f_{\theta_*}(n)e(\alpha n)
 \right|\,du \le
 \int_0^X
 \left|
 \sum_{u<n\le u+L_0}f_{\theta_*}(n)e(\alpha n)
 \right|\,du.
\end{align*}
It follows by Theorem~\ref{thm_MRT} that
\begin{align*}
 \sup_{\alpha\in\mathbb R}
 \sum_{1\le x\le Y}|S_{x,L_0}(\alpha)|
 &\ll YH\left(
 \ee^{-\Mpret(f_{\theta_*};X,Q_{L_0})/20}
 +\frac{\log\log L_0}{\log L_0}
 +(\log X)^{-1/700}\right),
\end{align*}
where $Q_{L_0} =\min\{(\log X)^{1/125},(\log L_0)^5\}$.

Since
\(
 \tfrac12\exp(G^{0.995})\le H\le\exp(2G),
\)
we have
\[
 G^{0.995}-\log2
 \le\log L_0
 \le2G+\log(2h+1)
 \le3G.
\]
Then
 \[\frac{\log\log L_0}{\log L_0} \le \frac{\log(3G)}{G^{0.995}-\log2} \ll G^{-0.995}\log G.\]

Since $N\ee^{-\eta}H<Y\le N\exp(2G)$, we have
\[
 N/3\le X\le N\exp(2G)+1\le N\exp(110G).
\]
So $X$ lies in the
range defining $M_*$ in \eqref{eq:pretentious-scale}, and $(\log X)^{-1/700}\ll(\log N)^{-1/700}$.

Let
\(
 Q_*=\min\{(\log X)^{1/125},(10G)^5\}.
\)
Since $\log L_0\le3G$, we have $Q_{L_0}\le Q_*$. So
\[
 \mathcal M(f_{\theta_*};X,Q_{L_0})
 \ge \mathcal M(f_{\theta_*};X,Q_*)
 \ge M_*,
\]
and hence \[
 \exp\!\left(-\frac{\mathcal M(f_{\theta_*};X,Q_{L_0})}{20}\right)
 \le \exp(-M_*/20).
\]

Therefore,
\begin{align*}
 \sup_{\alpha\in\mathbb R}
 \sum_{1\le x\le Y}|S_{x,L_0}(\alpha)|
 \ll
 XL_0\left(
 \exp(-M_*/20)
 +G^{-0.995}\log G
 +(\log N)^{-1/700}
 \right)\ll_h YH\varepsilon_*.
\end{align*}

If $\theta_*=\phi$, take $L_0=(2h+1)H$. Since
\[
 V_x(t)=S_{x,L_0}(-t),
\]
the required estimate \eqref{upper_bound_W} for $V_x$ follows immediately.

If $\theta_*=\theta$, take $L_0=H$ and expand the progression
indicator by additive characters:
\[
 \mathbf1_{n\equiv r_0\pmod q}
 =
 \frac1q\sum_{a=0}^{q-1}e\!\left(\frac{a(n-r_0)}q\right).
\]
This gives
\[
 U_x(t)
 =
 \frac1q\sum_{a=0}^{q-1}
 e\!\left(\frac{a(x-r_0)}q\right)
 S_{x,H}\!\left(t+\frac aq\right).
\]
The coefficients have total absolute value one. Hence
\begin{align*}
 \sup_t\sum_{1\le x\le Y}|U_x(t)|
 &\le
 \frac1q\sum_{a=0}^{q-1}
 \sup_t\sum_{1\le x\le Y}
 \left|S_{x,H}\!\left(t+\frac aq\right)\right|\\
 &=
 \sup_\alpha\sum_{1\le x\le Y}|S_{x,H}(\alpha)|\\
 &\ll_h YH\varepsilon_*.
\end{align*}
Thus, \eqref{upper_bound_W} holds for $U_x$.
\end{proof}

\begin{lemma}\label{lem:centering-reduction}
Let
\[
 \xi(s)=\exp\{-i(\theta+\phi)\Omega(s)\},
\]
and
\begin{align}\label{eq:phased-bin}
 \mathcal C_j=\frac1{T_j}\sum_{n\le T_j}
 \sum_{(d,v)\in\mathcal S_j}
 u(v)\xi(vd)\ind_{v\mid n}\ind_{n\equiv rvd\pmod q} \cdot
 \prod_{p\mid d}\left(\ind_{p\mid n}-\frac1p\right)
 f_\theta(n)f_\phi(n+hvd).
\end{align}
Then we have
\begin{equation}\label{eq:centering-error}
 \frac1{SV_*}\left|\sum_j\mathcal C_j
                   -S_0F_N^{h,q,r}(\theta,\phi)\right|
 \ll_{h,q}\eta+N^{-1}
  +\frac{2^J}{\eta}
       \left(G^{-1/20}+G^{1/4}\ee^{-M_*/20}\right).
\end{equation}
\end{lemma}

\begin{proof}
Expanding the product gives
\[
 \prod_{p\mid d}\left(\mathbf1_{p\mid n}-\frac1p\right)
 =
 \sum_{z\mid d}\frac{(-1)^{\omega(z)}}z
 \mathbf1_{\frac{d}{z}\mid n}.
\]
We separate the term $z=1$ from the terms $z>1$.

When $z=1$, the conditions $v\mid n$ and $d\mid n$ are equivalent
to $vd\mid n$. Substitute $n=vdt$. Since $(vd,q)=1$,
\[
 n\equiv rvd\pmod q
 \quad\Longleftrightarrow\quad
 t\equiv r\pmod q.
\]
Complete multiplicativity gives
\begin{align}
 \xi(vd)f_\theta(vdt)f_\phi(vd(t+h))
 &=
 \xi(vd)f_\theta(vd)f_\phi(vd)
 f_\theta(t)f_\phi(t+h) =f_\theta(t)f_\phi(t+h).
 \label{eq:full-phase}
\end{align}
Consequently, for fixed $j,d,v$, the contribution is
\[
 \frac{u(v)}{T_j}
 \sum_{\substack{t\le T_j/(vd)\\t\equiv r\pmod q}}
 f_\theta(t)f_\phi(t+h)
 =
 \frac{u(v)}{vd}
 F_{T_j/(vd)}^{h,q,r}(\theta,\phi).
\]

Recall that $T_j=Ne^{j\eta}$. For $(d,v)\in \mathcal S_j$, we have
\[
 e^{j\eta}\le vd<e^{(j+1)\eta},
 \qquad
 Ne^{-\eta}<\frac{T_j}{vd}\le N.
\]
To compare the corresponding averages, put
\[
 g(t)=\mathbf1_{t\equiv r\pmod q}
 f_\theta(t)f_\phi(t+h),
\]
so that $|g(t)|\le1$. For every real $0<Y\le N$,
\begin{align*}
 F_Y^{h,q,r}(\theta,\phi)-F_N^{h,q,r}(\theta,\phi)
 &=
 \left(\frac1Y-\frac1N\right)\sum_{t\le Y}g(t)
 -\frac1N\sum_{Y<t\le N}g(t).
\end{align*}
It follows that
\[
 \left|F_Y^{h,q,r}(\theta,\phi)-F_N^{h,q,r}(\theta,\phi)\right|
 \le
 \frac{N-Y}{N}+\frac{N-Y+1}{N}
 =
 \frac{2(N-Y)+1}{N}.
\]
If $Ne^{-\eta}<Y\le N$, then
\[
 \frac{N-Y}{N}<1-e^{-\eta}\le\eta,
\]
and hence
\[
 F_Y^{h,q,r}(\theta,\phi)
 =
 F_N^{h,q,r}(\theta,\phi)+O(\eta+N^{-1}).
\]
Summing over $j$ and recalling that
\[
 S_0=\sum_j\sum_{(d,v)\in\mathcal S_j}\frac{u(v)}{vd}
 \le SV_*,
\]
we conclude that the part of  $z=1$  contributes
\begin{equation}\label{eq:full-term}
 S_0F_N^{h,q,r}(\theta,\phi)
 +O\bigl(SV_*(\eta+N^{-1})\bigr).
\end{equation}

When $z>1$, write $d=d'z$. The divisibility conditions
are now $v\mid n$ and $d'\mid n$. Since $(v,d')=1$, substitute
$n=vd't$. The congruence becomes
\[
 vd't\equiv rvd'z\pmod q
 \quad\Longleftrightarrow\quad
 t\equiv rz\pmod q.
\]
Also, $n+hvd=vd'(t+hz)$, and complete multiplicativity gives
\begin{align}
 \xi(vd)f_\theta(vd't)f_\phi(vd'(t+hz))
 &=
 \xi(vd'z)f_\theta(vd')f_\phi(vd')
 f_\theta(t)f_\phi(t+hz)\notag\\
 &=\xi(z)f_\theta(t)f_\phi(t+hz).
 \label{eq:partial-phase}
\end{align}
Thus the contribution for fixed $j,v,d',z$ equals
\[
 \frac{(-1)^{\omega(z)}u(v)}{T_jz}
 \sum_{\substack{t\le T_j/(vd')\\t\equiv rz\pmod q}}
 \xi(z)f_\theta(t)f_\phi(t+hz).
\]

Let $I\subseteq\{1,\ldots,J\}$ be the nonempty set of indices
of the prime supplies contributing to $z$. Then
\[
 z=\prod_{i\in I}p_i,
 \qquad
 d'=\prod_{i\notin I}p_i,
 \qquad
 \omega(z)=|I|.
\]
Fix $j,I,v,d'$, and put
\[
 Y=\frac{T_j}{vd'}.
\]
The sign $(-1)^{\omega(z)}=(-1)^{|I|}$ is now fixed.
After factoring out $(-1)^{|I|}u(v)/(vd')$, the remaining
expression is
\[
 \frac1Y\sum_{t\le Y}\sum_z
 \frac{\xi(z)}z
 \mathbf1_{t\equiv rz\pmod q}
 f_\theta(t)f_\phi(t+hz),
\]
where $z$ contains exactly one prime from each supply indexed by
$I$, and the pair $(d'z,v)$ belongs to $\mathcal S_j$.

For fixed $j,v,d'$, we have
\[
 \frac{e^{j\eta}}{vd'}
 \le z<
 \frac{e^{(j+1)\eta}}{vd'}.
\]
Since $e^\eta<2$, this interval intersects at most two dyadic
intervals $[H,2H)$, with $H$ an integer power of two. Split the
labels further according to $z\bmod q$. There are at most $2q$
resulting pieces, and on each piece $rz\bmod q$ is a fixed
residue class $r_0$.

Let $\mathcal Z\subset[H,2H)$ be a nonempty resulting piece.
Every $z\in\mathcal Z$ contains at least one prime from
$\mathcal P$, and every such prime is at least
$\exp(G^{0.995})$. Moreover, $z\mid d$ for some
$d\in\mathcal D$, so $z\le d<\exp(2G)$. Therefore
\[
 \tfrac12\exp(G^{0.995})\le H\le\exp(2G).
\]
By $(d'z,v)\in \mathcal S_j$,
\[
 Ne^{-\eta}z<Y\le Nz
 \qquad(z\in\mathcal Z).
\]
It follows that
\[
 Y>Ne^{-\eta}H,
 \qquad
 Y\le Ne^{2G},
 \qquad
 \frac HY<\frac{e^\eta}{N}.
\]
Thus all the hypotheses of Lemma~\ref{lem:rough-shifts} hold.
Applying that lemma with $c_z=\xi(z)$ gives
\[
 \left|
 \frac1Y\sum_{t\le Y}\sum_{z\in\mathcal Z}
 \frac{\xi(z)}z
 \mathbf1_{t\equiv r_0\pmod q}
 f_\theta(t)f_\phi(t+hz)
 \right|
 \ll_h
 G^{-21/20}+G^{-3/4}\varepsilon_*+N^{-1}.
\]

For fixed $j$ and nonempty $I$, 
\begin{align*}
\Big|\sum_{\text{admissible }v,d'} (-1)^{|I|}u(v)/(vd') \Big| \leq  \sum_{\text{admissible }v,d'}\frac{u(v)}{vd'}
 \le
 \left(\sum_{v\in\mathcal V}\frac{u(v)}v\right)
 \prod_{i\notin I}\left(\sum_{p\in\mathcal P_i}\frac1p\right)=S\prod_{i\notin I}V_i
 \le SV_*.
\end{align*}
Here
\[
 \sum_{v\in\mathcal V}\frac{u(v)}v
 =\prod_{p\in\mathcal Q}\left(1+\frac4p\right)=S,
\]
and the last inequality follows from $V_i\ge W\ge1$.

There are at most $2^J$ nonempty choices of $I$, there are
$O(G/\eta)$ many $\mathcal S_j$'s, and each fixed choice gives at most $2q$
pieces. Hence the total contribution for the part $z>1$,
after division by $SV_*$, is
\begin{align*}
 &\ll_{h,q}
 2^J\frac{G}{\eta}
 \left(G^{-21/20}+G^{-3/4}\varepsilon_*+N^{-1}\right)\\
 &= \frac{2^J}{\eta}
 O_{h,q}\left(
 G^{-1/20}+G^{1/4}\varepsilon_*+\frac GN
 \right).
\end{align*}

Finally, by \eqref{eq:epsilon-star},
\[
 G^{1/4}\varepsilon_*
 =
 G^{1/4}e^{-M_*/20}
 +G^{-0.745}\log G
 +G^{1/4}(\log N)^{-1/700}.
\]
For sufficiently large $G$,
\[
 G^{-0.745}\log G\ll G^{-1/20}.
\]
Also, the scale condition $\log N\ge G^{A_0}$ implies
\[
 G^{1/4}(\log N)^{-1/700}
 \le G^{1/4-A_0/700}
 \le G^{-1/20}
\]
provided $A_0\ge210$. Increasing $A_0$ if necessary, we also have
\[
 \frac GN\le G\exp(-G^{A_0})\ll G^{-1/20}.
\]
Thus the normalized total contribution for the part $z>1$ is
\[
 \ll_{h,q}
 \frac{2^J}{\eta}
 \left(G^{-1/20}+G^{1/4}e^{-M_*/20}\right).
\]
Combining this with \eqref{eq:full-term}, we obtain
\[
 \frac1{SV_*}
 \left|
 \sum_j\mathcal C_j-S_0F_N^{h,q,r}(\theta,\phi)
 \right|
 \ll_{h,q}
 \eta+N^{-1}
 +\frac{2^J}{\eta}
 \left(G^{-1/20}+G^{1/4}e^{-M_*/20}\right),
\]
which proves \eqref{eq:centering-error}.
\end{proof}

\begin{proof}[Proof of Proposition~\ref{prop:progression}]
By \cite[Lemma 4.6, Proposition 8.4]{OAI}, we have
\begin{equation}\label{eq:deletion-estimate}
 \frac1{SV_*}\sum_j|\mathcal C_j|
 \ll_{h,q}2^J(K^{-1}+G^{-100}+\ee^{-2WJ})+\ee^{-G^{0.9}} +\frac{G}{\eta SV_*}\Big(\frac{SK(C_3\sqrt W)^J}{G} + \ee^{-G} \Big).
\end{equation}

Combining Lemma~\ref{lem:centering-reduction} and
\eqref{eq:deletion-estimate}, by
$V_*\ge W^J$ and $S_0\ge SV_*/2$, we have
\begin{align}\label{eq:complete-graph-error}
 |F_N^{h,q,r}(\theta,\phi)|\ll_{h,q}{}&\eta+N^{-1}
 +\frac{2^J}{\eta}\left(G^{-1/20}+G^{1/4}\ee^{-M_*/20}\right)\notag\\
 &+2^J(K^{-1}+G^{-100}+\ee^{-2WJ})
 +\frac K\eta\left(\frac{C_3}{\sqrt W}\right)^J
 +\ee^{-G^{0.9}}.
\end{align}

Recall that
\[
 J=\lfloor b\log G\rfloor,\qquad
 \eta=e^{-J},\qquad K=e^{4J},\qquad
 b=\frac{\delta_0}{6W}.
\]
Notice that for every fixed $a>0$, we have
\begin{equation}\label{parameter_W_b}
 e^{aJ}\le G^{ab},
 \qquad
 e^{-aJ}\le e^aG^{-ab}.
\end{equation}
Let $W$ be sufficiently large such that $b$ is sufficiently small.  Using \eqref{parameter_W_b} for each term of \eqref{eq:complete-graph-error}, we get that
\begin{equation}\label{eq:graph-M-bound}
 |F_N^{h,q,r}(\theta,\phi)|
 \ll_{h,q}G^{-b}+Ge^{-M_*/20}.
\end{equation}

Recall that
\[
 M_*=
 \inf_{N/3\le Y\le Ne^{110G}}
 \mathcal M\left(
 f_{\theta_*};Y,
 \min\{(\log Y)^{1/125},(10G)^5\}
 \right),
\]
where
\[
 |\theta_*|=\max\{|\theta|,|\phi|\},
 \qquad L=\log\log N.
\]
By Lemma~\ref{lem:distance},
\[
 \mathcal M\left(
 f_{\theta_*};Y,
 \min\{(\log Y)^{1/125},(10G)^5\}
 \right)
 \ge c_1\theta_*^2\log\log Y-C_1
\]
for absolute constants $c_1,C_1>0$.
For $N/3\le Y\le Ne^{110G}$, we have  $\log\log Y=L+O(1).$ So
\[
 M_*\ge c_1\theta_*^2L-C_2.
\]
Substituting this into \eqref{eq:graph-M-bound} proves
\[
 |F_N^{h,q,r}(\theta,\phi)|
 \ll_{h,q}
 G^{-b}
 +G\exp\left(
 -c\max\{\theta^2,\phi^2\}\log\log N
 \right)
\]
for an absolute constant $c>0$.
\end{proof}

To pass to the two affine forms, set
\[
 q=a_1a_2,\qquad c_0=\min(a_2b_1,a_1b_2),\qquad
 h=|\delta|,\qquad N'=qN+c_0.
\]
Take $(\theta',\phi')=(\theta,\phi)$ if $\delta>0$, and
$(\theta',\phi')=(\phi,\theta)$ if $\delta<0$. If
\[
 P_{\theta',\phi'}(Y)
 =\sum_{\substack{n\le Y\\n\equiv c_0\pmod q}}
 f_{\theta'}(n)f_{\phi'}(n+h),
\]
complete multiplicativity gives the exact identity
\begin{equation}\label{eq:affine-dilation}
 C_N^\ell(\theta,\phi)
 =\frac{\exp\{-i\theta\Omega(a_2)-i\phi\Omega(a_1)\}}{N}
 \bigl(P_{\theta',\phi'}(N')-P_{\theta',\phi'}(c_0)\bigr).
\end{equation}
Since $N'/N\ll_{\ell_1,\ell_2}1$,
$P_{\theta',\phi'}(c_0)=O_{\ell_1,\ell_2}(1)$, and
$\log\log N'=L+O_{\ell_1,\ell_2}(1/\log N)$,
\eqref{eq:progression-phase-bound} implies
\begin{equation}\label{eq:affine-high-basic}
 |C_N^\ell(\theta,\phi)|
 \ll_{\ell_1,\ell_2}G^{-b}
       +G\exp\{-cL\max(\theta^2,\phi^2)\}.
\end{equation}

\begin{corollary}\label{cor:high-frequency}
There is an absolute constant $D>0$ such that, with
\begin{equation}\label{eq:rho}
 \rho=\sqrt{\frac{D\log L}{L}},
\end{equation}
one has
\begin{equation}\label{eq:high-frequency}
 \sup_{[-\pi,\pi]^2\setminus[-\rho,\rho]^2}
 |\Delta_N(\theta,\phi)|\ll_{\ell_1,\ell_2}L^{-14}.
\end{equation}
\end{corollary}

\begin{proof}
Recall that
\[
 \Delta_N(\theta,\phi)
 =C_N^\ell(\theta,\phi)-A_N(\theta)A_N(\phi).
\]

Suppose 
$\max\{|\theta|,|\phi|\}\ge\rho$. Choose a fixed $B>0$ such that $Bb\ge14$, and set $G=L^B$.
Then $G^{-b}=L^{-Bb}\le L^{-14}.$  Choose a fixed $D>0$ such that $cD\ge B+14$, then
\[
 \max\{\theta^2,\phi^2\}L
 \ge \rho^2L
 =D\log L.
\]
Consequently,
\[
 G\exp\left(-c\max\{\theta^2,\phi^2\}L\right)
 \le L^B e^{-cD\log L}
 =L^{B-cD}
 \le L^{-14}.
\]
Combining the two terms gives
\[
 |C_N^\ell(\theta,\phi)|
 \ll_{\ell_1,\ell_2}L^{-14}
 \qquad
 \bigl(\max\{|\theta|,|\phi|\}\ge\rho\bigr).
\]

Recall that
\[
 A_N(t)=\frac1N\sum_{n\le N}f_t(n),
\]
take $h=q=1$, $r=0$, and phases $(t,0)$ in
Proposition~\ref{prop:progression}. Then
\[
 F_N^{1,1,0}(t,0)
 =\frac1N\sum_{n\le N}f_t(n)f_0(n+1)
 =A_N(t).
\]
Hence
\[
 |A_N(t)|\ll G^{-b}+Ge^{-ct^2L}.
\]

For $|t|\ge\rho$, by the argument above, we have
\(
 |A_N(t)|\ll L^{-14}.
\)
For arbitrary $t$, the trivial bound is
\(
 |A_N(t)|\le1.
\)
Thus, on the region $\max\{|\theta|,|\phi|\}\ge\rho$, 
\[
 |A_N(\theta)A_N(\phi)|\ll L^{-14}.
\]

By the triangle inequality, we obtain \eqref{eq:high-frequency}.
\end{proof}

\section{Small frequencies}

\begin{proposition}\label{prop:small-frequency}
There is an absolute constant $c>0$ such that, uniformly for
$|\theta|,|\phi|\le L^{-1/3}$,
\begin{equation}\label{eq:small-pointwise}
 |\Delta_N(\theta,\phi)|
 \ll_{\ell_1,\ell_2}(|\theta|+|\phi|)\log L+|\theta\phi|
                +\ee^{-cL\log L}+N^{-1/2}.
\end{equation}
\end{proposition}

\begin{proof}
For $\ell_i(n)=a_in+b_i, \delta=a_1b_2-a_2b_1\ne0, i=1,2$ and $L=\log\log N$, 
define
\[
 \mathcal B=\{p:p\mid a_1a_2|\delta|\},
 \qquad y=N^{1/L^2},
 \qquad
 S(n)=\sum_{\substack{p\le y\\p\notin\mathcal B}}
       \mathbf1_{p\mid n}.
\]
 Put
\(
 R(n)=\Omega(n)-S(n).
\)
Then
\[
 R(n)
 =
 \sum_{\substack{p>y\ \text{or}\ p\in\mathcal B}}
       \mathbf1_{p\mid n}
 +\sum_p\sum_{j\ge2}\mathbf1_{p^j\mid n}.
\]

Let
\(
 C_*=\max(1,a_1+b_1,a_2+b_2), M=C_*N.
\)
Then $\ell_i(n)\leq M$ for  $1\leq n \leq N$, and
we obtain
\begin{align}
 \frac1N\sum_{n\le N}R(\ell_i(n))
 & \le\frac1N\sum_{m\le M}R(m) \le
 C_*\left(
 \sum_{y<p\le M}\frac1p
 +\sum_{p\in\mathcal B}\frac1p
 +\sum_p\sum_{j\ge2}\frac1{p^j}
 \right).
 \label{eq:affine-truncation-intermediate}
\end{align}
The last two sums are bounded. By Mertens' estimate,
\[
 \sum_{y<p\le M}\frac1p
 =\log\log M-\log\log y+O(1).
\]
By
\(
 \log\log y=L-2\log L\)
 and
 \(
 \log\log M=L+O_{\ell_1,\ell_2}(1).
\)
we get that
\begin{equation}\label{eq:affine-truncation}
 \frac1N\sum_{n\le N}R(\ell_i(n))
 \ll_{\ell_1,\ell_2}\log L.
\end{equation}
The same argument gives
\[
 \frac1N\sum_{n\le N}R(n)
 \ll_{\ell_1,\ell_2}\log L.
\]

For real $u,v$, the inequality
$|e^{iu}-e^{iv}|\le|u-v|$ implies
\begin{align*}
 &\left|
 e^{i\theta\Omega(\ell_1(n))+i\phi\Omega(\ell_2(n))}
 -
 e^{i\theta S(\ell_1(n))+i\phi S(\ell_2(n))}
 \right|\le
 |\theta|R(\ell_1(n))+|\phi|R(\ell_2(n)).
\end{align*}
So
\begin{equation}\label{truncation_1}
\frac1N\sum_{n\le N}\left|
 e^{i\theta\Omega(\ell_1(n))+i\phi\Omega(\ell_2(n))}
 -
 e^{i\theta S(\ell_1(n))+i\phi S(\ell_2(n))}
 \right|\ll ( |\theta| + |\phi| )\log L.
\end{equation}

Let
\[
A_N(t)=\frac1N\sum_{n\le N}e^{it\Omega(n)},
\qquad
B_N(t)=\frac1N\sum_{n\le N}e^{itS(n)}.
\]
Then
\[
\begin{aligned}
|A_N(t)-B_N(t)|
&\le\frac1N\sum_{n\le N}
|e^{it\Omega(n)}-e^{itS(n)}|\le\frac{|t|}{N}\sum_{n\le N}R(n)\ll_{\ell_1,\ell_2}|t|\log L.
\end{aligned}
\]
So
\begin{equation}\label{truncation_2}
|A_N(\theta)-B_N(\theta)|
\ll_{\ell_1,\ell_2}|\theta|\log L,
\qquad
|A_N(\phi)-B_N(\phi)|
\ll_{\ell_1,\ell_2}|\phi|\log L.
\end{equation}

For each $p\le y$ with $p\notin\mathcal B$, independently choose
a uniform residue $U_p\in\mathbb Z/p\mathbb Z$ and define
\[
 X_p=\mathbf1_{a_1U_p+b_1\equiv0\pmod p},
 \qquad
 Y_p=\mathbf1_{a_2U_p+b_2\equiv0\pmod p}.
\]
Since $p\nmid a_1a_2$, each congruence has exactly one root.
Therefore
\[
 \mathbb P(X_p=1)=\mathbb P(Y_p=1)=\frac1p,
 \qquad X_pY_p=0.
\]
Equivalently,
\[
 \mathbb P((X_p,Y_p)=(1,0))=\frac1p,\qquad
 \mathbb P((X_p,Y_p)=(0,1))=\frac1p,
\]
and
\[
 \mathbb P((X_p,Y_p)=(0,0))=1-\frac2p.
\]
Put
\[
 S_1^*=\sum_{\substack{p\le y\\p\notin\mathcal B}}X_p,
 \qquad
 S_2^*=\sum_{\substack{p\le y\\p\notin\mathcal B}}Y_p.
\]
The vectors $(X_p,Y_p)$ are independent across distinct primes.

A degree-$k$ monomial in the expansion of
\[
 \bigl(\theta S(\ell_1(n))+\phi S(\ell_2(n))\bigr)^k
\]
is a scalar multiple of a product
\[
 \prod_{a=1}^{k}\mathbf1_{p_a\mid\ell_{i_a}(n)},
 \qquad i_a\in\{1,2\}.
\]
There are at most $(2y)^k$ ordered monomials. If
$|\theta|,|\phi|\le\rho_0$, each scalar coefficient has
absolute value at most $\rho_0^k$. Notice that
\[
 \left|
 \frac1N\#\{n\le N:n\equiv a\pmod d\}-\frac1d
 \right|
 \le\frac1N.
\]
for any $d\ge2$. Hence, for $k\ge1$,
\begin{align}
 &\left|
 \frac1N\sum_{n\le N}
 \bigl(\theta S(\ell_1(n))+\phi S(\ell_2(n))\bigr)^k
 -
 \mathbb E(\theta S_1^*+\phi S_2^*)^k
 \right|
 \le\frac{(2\rho_0y)^k}{N}.
 \label{eq:moment-comparison}
\end{align}

Set $Z^*=S_1^*+S_2^*$. Since $X_p+Y_p$ equals one
with probability $2/p$ and zero otherwise,
\[
 \mathbb Ee^{X_p+Y_p}=1+\frac{2(e-1)}p.
\]
Independence across primes therefore gives
\[
 \mathbb Ee^{Z^*}
 =
 \prod_{\substack{p\le y\\p\notin\mathcal B}}
 \left(1+\frac{2(e-1)}p\right)
 \le
 \exp\left(2(e-1)\sum_{p\le y}\frac1p\right)
 \le e^{CL}
\]
for an absolute constant $C$.

Take
\[
 \rho_0=L^{-1/3},\qquad m=\lceil L\rceil,
 \qquad
 P_{m-1}(x)=\sum_{k=0}^{m-1}\frac{(ix)^k}{k!}.
\]
For real $x$, we have
\[
 |e^{ix}-P_{m-1}(x)|\le\frac{|x|^m}{m!}.
\]
Let $X^*=\theta S_1^*+\phi S_2^*$. Since
\(
 |\theta S_1^*+\phi S_2^*|\le\rho_0Z^*,
\)
for the remainder we have
\begin{align*}
\mathbb E\left|e^{iX^*}-P_{m-1}(X^*)\right|
\leq  \frac{\rho_0^m}{m!}\mathbb E(Z^*)^m
 \le \rho_0^m\mathbb Ee^{Z^*}\le
 \exp\left(-\frac m3\log L+CL\right)
 \le e^{-cL\log L}
\end{align*}
for sufficiently large $L$ for some constant $c>0$.

For the integer remainder, let
\[
X_n=\theta S(\ell_1(n))+\phi S(\ell_2(n)),
\qquad
P_{m-1}(x)=\sum_{k=0}^{m-1}\frac{(ix)^k}{k!}.
\]
By the above estimate, we have
\[
\begin{aligned}
\frac1N\sum_{n\le N}
\left|e^{iX_n}-P_{m-1}(X_n)\right|
&\le
\frac{\rho_0^m}{m!}\mathbb E(Z^*)^m
+\frac{\rho_0^m(2y)^m}{m!N}\\
&\le e^{-cL\log L}+\frac{(2y)^m}{N}.
\end{aligned}
\]

By linearity,
\[
 \frac1N\sum_{n\le N}P_{m-1}(X_n)
 -\mathbb EP_{m-1}(X^*)
 =
 \sum_{k=0}^{m-1}\frac{i^k}{k!}
 \left(
 \frac1N\sum_{n\le N}X_n^k-\mathbb E(X^*)^k
 \right).
\]
The triangle inequality  gives
\[
 \left|
 \frac1N\sum_{n\le N}P_{m-1}(X_n)
 -\mathbb EP_{m-1}(X^*)
 \right|
 \le
 \sum_{k=0}^{m-1}\frac1{k!}
 \left|
 \frac1N\sum_{n\le N}X_n^k-\mathbb E(X^*)^k
 \right|.
\]
For every $k\ge1$, the moment comparison \eqref{eq:moment-comparison} yields
\[
 \left|
 \frac1N\sum_{n\le N}X_n^k-\mathbb E(X^*)^k
 \right|
 \le\frac{(2\rho_0y)^k}{N}.
\]
For sufficiently large $N$, we have $\rho_0=L^{-1/3}\le1$
and $2y\ge1$.  Hence, 
\[
 \left|
 \frac1N\sum_{n\le N}P_{m-1}(X_n)
 -\mathbb EP_{m-1}(X^*)
 \right|
 \le
 \frac1N\sum_{k=0}^{m-1}\frac{(2\rho_0y)^k}{k!} \le\frac{m(2y)^m}{N} \ll N^{-1/2}
\]
for $y=N^{1/L^2}$ and $m=\lceil L\rceil$.

Thus, uniformly for $|\theta|,|\phi|\le\rho_0$,
\begin{align}
 \frac1N\sum_{n\le N}
 e^{i\theta S(\ell_1(n))+i\phi S(\ell_2(n))}
 ={}&
 \mathbb Ee^{i\theta S_1^*+i\phi S_2^*} +O\left(e^{-cL\log L}+N^{-1/2}\right).
 \label{eq:char-model}
\end{align}

Similarly,
\begin{equation}\label{eq:char-model_1}
 \frac1N\sum_{n\le N}e^{itS(n)}
 =
 \mathbb Ee^{itS_i^*}
 +O\left(e^{-cL\log L}+N^{-1/2}\right),
 \qquad |t|\le\rho_0,
\end{equation}
for either $i=1$ or $i=2$.

Write
\[
 u=e^{i\theta}-1,\qquad v=e^{i\phi}-1.
\]
Then
\[
 J_p=\mathbb Ee^{i\theta X_p+i\phi Y_p}
 =1+\frac{u+v}{p},
\]
and
\[
 P_p=\mathbb E e^{i\theta X_p} \mathbb E e^{i\phi Y_p}=
 \left(1+\frac up\right)\left(1+\frac vp\right)
 =1+\frac{u+v}{p}+\frac{uv}{p^2}.
\]
Hence
\[
 |J_p-P_p|
 =\frac{|uv|}{p^2}
 \le\frac{|\theta\phi|}{p^2}.
\]
Both $|J_p|$ and $|P_p|$ are at most one,  telescoping the finite products gives
\[
 \left|\prod_pJ_p-\prod_pP_p\right|
 \le\sum_p|J_p-P_p|
 \le|\theta\phi|\sum_p\frac1{p^2}
 \ll|\theta\phi|.
\]
That is,
\begin{equation}\label{eq:char-model_2}
 \left|
 \mathbb Ee^{i\theta S_1^*+i\phi S_2^*}
 -
 \mathbb Ee^{i\theta S_1^*}\,
 \mathbb Ee^{i\phi S_2^*}
 \right|
 \ll|\theta\phi|.
\end{equation}

Combining \eqref{eq:char-model}, \eqref{eq:char-model_1} and \eqref{eq:char-model_2}, we obtain
\[
 \left|
 \frac1N\sum_{n\le N}
 e^{i\theta S(\ell_1(n))+i\phi S(\ell_2(n))}
 -
 B_N(\theta)B_N(\phi)
 \right|
 \ll|\theta\phi|+e^{-cL\log L}+N^{-1/2}.
\]

Finally, the truncation estimates \eqref{truncation_1} and \eqref{truncation_2} imply
\[
 \left|
 C_N^\ell(\theta,\phi)
 -
 \frac1N\sum_{n\le N}
 e^{i\theta S(\ell_1(n))+i\phi S(\ell_2(n))}
 \right|
 \ll_{\ell_1,\ell_2}(|\theta|+|\phi|)\log L,
\]
and
\[
 |A_N(\theta)A_N(\phi)-B_N(\theta)B_N(\phi)|
 \ll_{\ell_1,\ell_2}(|\theta|+|\phi|)\log L.
\]
Therefore,
\[
 |\Delta_N(\theta,\phi)|
 \ll_{\ell_1,\ell_2}
 (|\theta|+|\phi|)\log L
 +|\theta\phi|
 +e^{-cL\log L}+N^{-1/2},
 \qquad
 |\theta|,|\phi|\le L^{-1/3}.
\]
This proves \eqref{eq:small-pointwise}.
\end{proof}

By Proposition~\ref{prop:small-frequency}, we immediately obtain the following estimate for the square integral over small frequencies.

\begin{corollary}\label{cor:small-L2}
For the value of $\rho=\sqrt{\frac{D\log L}{L}}$ in \eqref{eq:rho},
\begin{equation}\label{eq:small-L2}
 \int_{[-\rho,\rho]^2}|\Delta_N(\theta,\phi)|^2
       \dd\theta\dd\phi
 \ll_{\ell_1,\ell_2}\rho^4(\log L)^2.
\end{equation}
\end{corollary}

By the proof of \cite[Lemma 2.3]{CharamarasRichter2025}, we have

\begin{corollary}\label{cor:affine-tail}
Set
\begin{equation}\label{eq:affine-R}
 R=\sqrt{L\log L}.
\end{equation}
Then, for all sufficiently large $N$,
\begin{equation}\label{eq:affine-tail}
 \tau_N(R)+\tau_{N,1}^\ell(R)+\tau_{N,2}^\ell(R)
 \ll_{\ell_1,\ell_2}L^{-1/2}.
\end{equation}
\end{corollary}

\begin{proof}[Proof of Theorem~\ref{thm:main}]
Choose $\rho$ as in Corollary~\ref{cor:high-frequency}. That corollary
and Corollary~\ref{cor:small-L2} imply
\begin{equation}\label{eq:global-L2}
 \norm{\Delta_N}_{L^2([-\pi,\pi]^2)}
 \ll_{\ell_1,\ell_2}\rho^2\log L+L^{-14}.
\end{equation}
Then Lemma
\ref{lem:fourier-reduction} and Corollary~\ref{cor:affine-tail} give
\begin{align*}
 \sum_{k,j}|d_N(k,j)|
 &\ll_{\ell_1,\ell_2}
       R\rho^2\log L+RL^{-14}+L^{-1/2}\\
 &\ll_{\ell_1,\ell_2}
       \sqrt{L\log L}\,\frac{\log L}{L}\,\log L\\
 &\ll_{\ell_1,\ell_2}\frac{(\log L)^{5/2}}{\sqrt L}.
\end{align*}
This proves \eqref{eq:main-tv}. The difference of the two averages
in \eqref{eq:main-test} is exactly
\[
 \sum_{k,j\ge0}a(k,j)d_N(k,j),
\]
whose modulus is at most $\norm a_\infty\sum_{k,j}|d_N(k,j)|$.
This proves Theorem~\ref{thm:main}.
\end{proof}

\end{document}